\documentclass{amsart}
\usepackage{amsfonts}
\usepackage{amssymb, latexsym, amsmath, epsfig}
\renewcommand{\baselinestretch}{1.2}
\begin{document}
%\begin{CJK}{GBK}{song}
\renewcommand{\baselinestretch}{1.0}
\newtheorem{theorem}{Theorem}[section]
\newtheorem{lemma}[theorem]{Lemma}
\newtheorem{corollary}[theorem]{Corollary}
\newtheorem{prop}[theorem]{Proposition}
\newtheorem{defi}[theorem]{Definition}
\newtheorem{cons}[theorem]{Consequence}
\renewcommand{\theequation}{\arabic{section}.\arabic{equation}}
\newtheorem{definition}[theorem]{Definition}
\newtheorem{example}[theorem]{Example}
\newtheorem{remark}[theorem]{Remark}
\newtheorem{assumption}[theorem]{Assumption}
\newtheorem{algo}[theorem]{Algorithm}
 \def\ad#1{\begin{aligned}#1\end{aligned}}  \def\b#1{\mathbf{#1}}
\def\a#1{\begin{align*}#1\end{align*}} \def\an#1{\begin{align}#1\end{align}}
\def\e#1{\begin{equation}#1\end{equation}} \def\d{\nabla\cdot}
\def\p#1{\begin{pmatrix}#1\end{pmatrix}} \def\c{\operatorname{curl}}
 \def\vc{\nabla\times } \numberwithin{equation}{section}
\def\boxit#1{\vbox{\hrule height1pt \hbox{\vrule width1pt\kern1pt
     #1\kern1pt\vrule width1pt}\hrule height1pt }}
 \def\lab#1{\boxit{\small #1}\label{#1}}
  \def\mref#1{\boxit{\small #1}\ref{#1}}
 \def\meqref#1{\boxit{\small #1}\eqref{#1}}

   \def\lab#1{\label{#1}} \def\mref#1{\ref{#1}} \def\meqref#1{\eqref{#1}}

\def\T{{\mathcal T}}\def\Th{{\mathcal T}_h}

\title[Flux-conserving   FEM  ]
   {Flux-conserving finite element methods}
\author{Shangyou Zhang}
\address{Department of Mathematical Sciences, University of Delaware,
    Newark, DE 19716, USA.  szhang@udel.edu }
    \author{Zhimin Zhang}
\address{Department of Mathematics, Wayne State University, Detroit,
   MI 48202, USA. ag7791@wayne.edu }
\thanks{This author is partially supported by the US National Science Foundation through grant
DMS-0612908, the Ministry of Education of China through the Changjiang Scholars program,
and Guangdong Provincial Government of China through the
``Computational Science Innovative Research Team" program.}
\author{Qingsong Zou}
\address{Department of Scientific Computation and Computers' Applications and Guangdong Province Key Laboratory of Computational Science, 
  Sun Yat-sen University, Guangzhou, 510275, China.
    mcszqs@mail.sysu.edu.cn }
\begin{abstract}
We analyze the flux conservation property of the
      finite element method.
It is shown that the finite element solution
  does approximate the flux locally in the optimal order, i.e.,
  the same order as that of the nodal interpolation operator.
We propose two methods, post-processing the finite element solutions locally.
The new solutions, remaining as optimal-order solutions,
   are flux-conserving elementwise.
In one of our methods, the processed solution also satisfies the
   original finite element equations.
While the high-order finite volume schemes
 are still under construction, our methods produce
 finite-volume-like finite element solution of any order.
In particular,  our methods avoid solving non-symmetric finite volume
   equations.
Numerical tests in 2D and 3D verify our findings.

%\end{AMS}

% \date{January 1, 2001 and, in revised form, June 22, 2001.}


\vskip .3cm

%\begin{keywords}

\end{abstract}

\maketitle

\section{Introduction}
The finite element method (FEM) is a most popular method in solving
  partial differential equations, cf. \cite{Brenner}.
Due to its local conservation property of flux,
     the finite volume method (FVM) is also
      used in a wide range of computation,
    especially in  computational fluid dynamics, cf. \cite{LeVeque2002,Shu}.
However, the mathematical theory on FVM
   (cf., \cite{Eymard, LeVeque2002, LiChenWu2000})
   has not been fully developed, at least,
       not as satisfactory as that for FEM.
Low order FVM theory ($P_1$ and some $P_2$ methods) is well established,
    see e.g., \cite{BankRRoseD1987,  Cai1, Cai3, EwingLinLin2002,
     Hackbusch1988,LiChenWu2000}.
Both the design and analysis on high order, symmetric FVMs are still
   under investigation,
   see \cite{ Cai.Z_Park.M2003, chen,  ChenWuXu2011,
    Liebau1996,PlexousakisZouraris2004, Tian.Chen_1991,
  Xu.J;Zou.Q2009}.

In this paper, we seek flux-conserving solutions
    in a completely different way.
The basic idea is to
   do a post processing on the finite element solution so that
   the new solution is flux conserving elementwise.
To this end, we first analyze the flux-conservation property
     of the finite element solution.
It is shown that the order of approximation of finite element in flux
   is optimal,  the same order as that of nodal interpolation.
The order is $O(h^{k-1})$ for the $k$-th degree finite element.
That is,  there is no convergence in total flux when using linear ($P_1$)
   finite element.
This is also confirmed numerically, both in 2D and 3D.
To overcome this shortcoming of finite element method,
  we  correct the flux error locally
    on each element by bubble functions, cf. Figure \mref{bubble13}.
In one method, we use the same order bubbles,  cubic bubbles in 2D and
    quartic bubbles in 3D, for any order finite element method.
In the second method,  we use degree-$(k+2)$ orthogonal bubbles
    for the $k$-th degree finite element solutions.
In the latter method,  the post-processed finite element solution
   still satisfies the original finite element equations.
In both methods, the post-processed finite element solution remains
   as an optimal-order solution,   in both $H^1$ and $L^2$ norms.

\begin{figure}[htb]
   \begin{center}\setlength{\unitlength}{1in}
 \begin{picture}(4.7,1)(0,.1)
 \put(0,0){\includegraphics[scale=0.28]{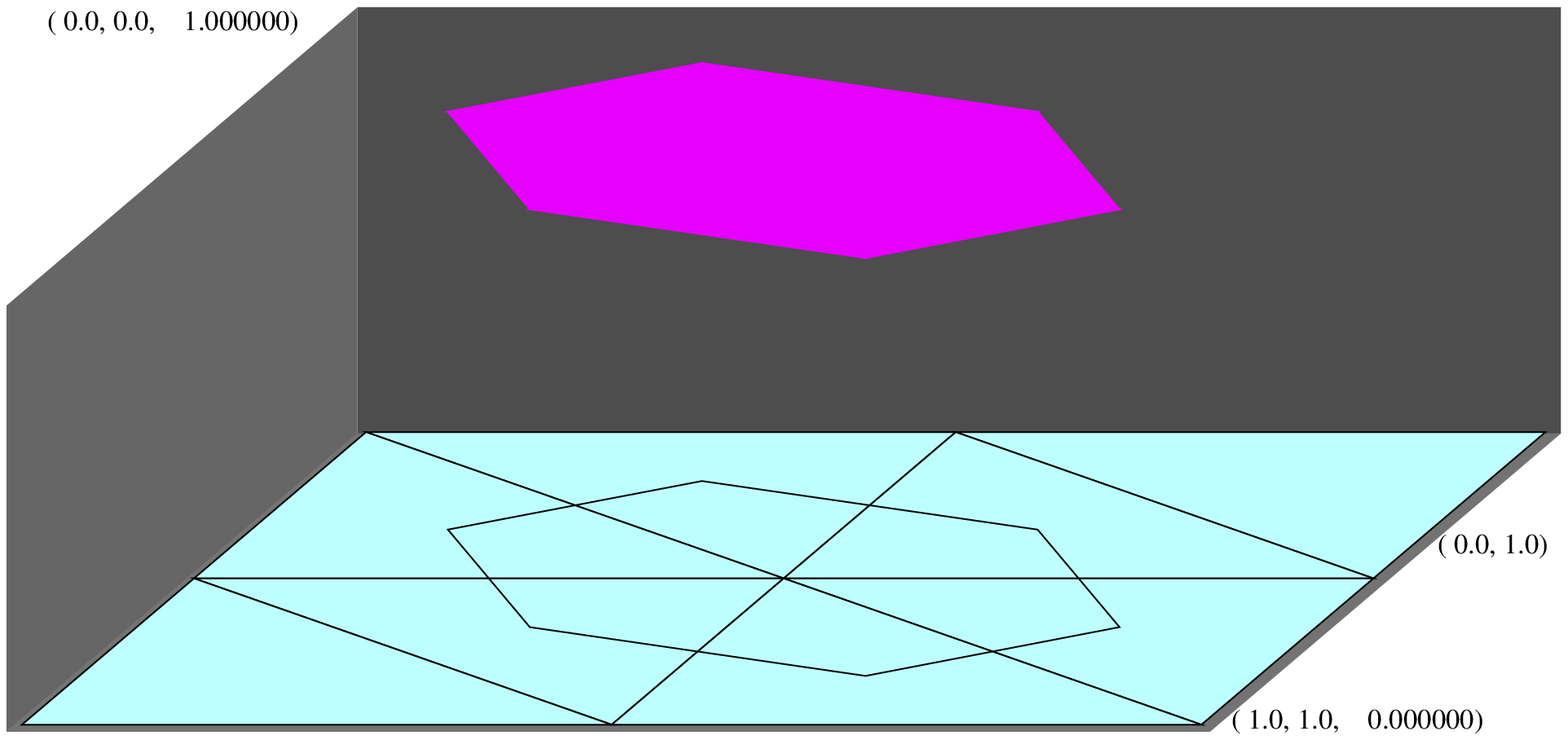}}
 \put(2.4,0){\includegraphics[scale=0.28]{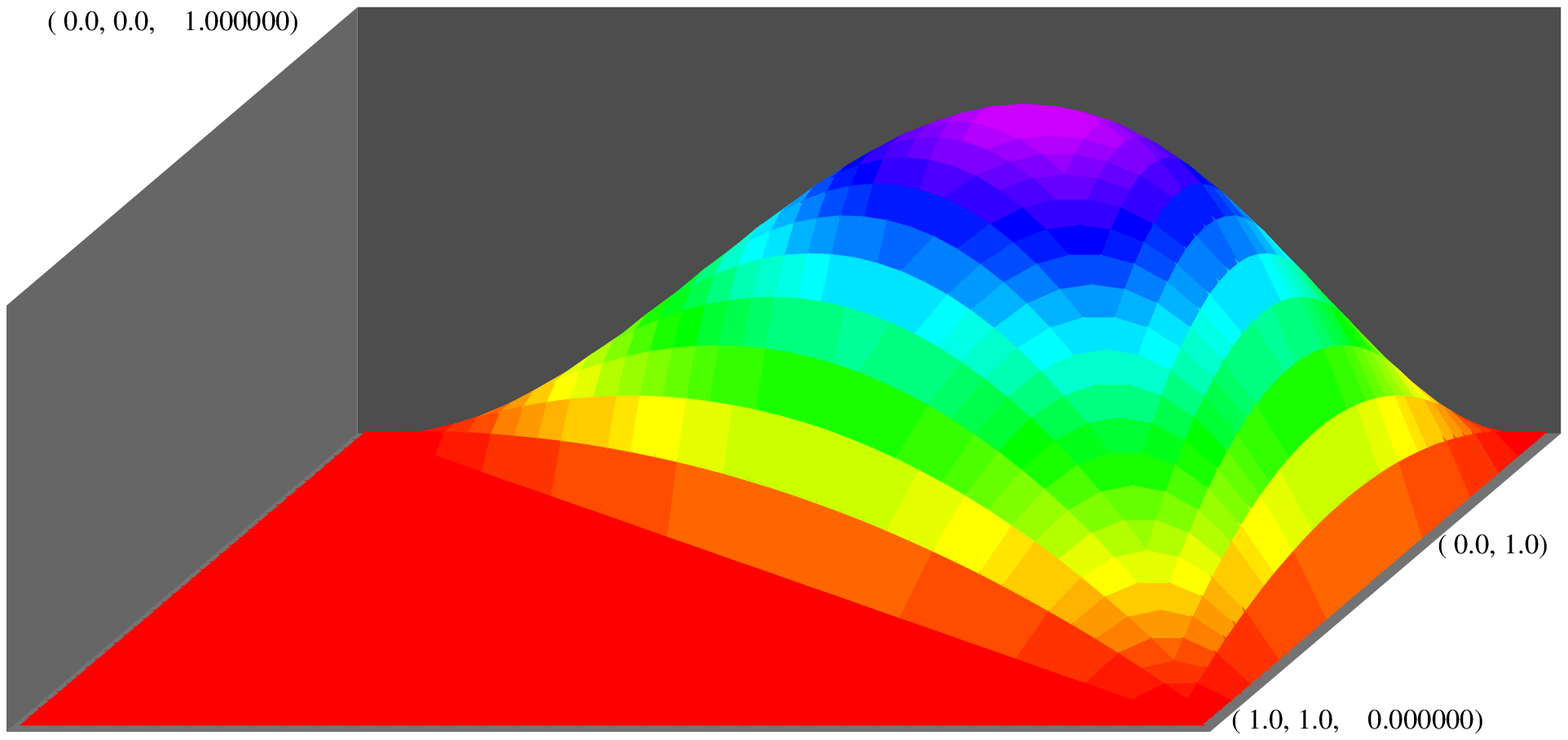}}
 \end{picture}
   \end{center}
\caption{A finite volume test function, and a cubic bubble function.}
    \lab{bubble13}
\end{figure}

We note that the finite volume solution is flux-conserving on each
   dual cell, i.e., control-volume, see the left graph in Figure
	\mref{bubble13}.
However for most high-order FVM solutions, the approximate theory has
    not been established  cf. \cite{chen,ChenWuXu2011,xu,Xu.J;Zou.Q2009}.
Also, other than $P_1$ FVM for constant coefficient problems,
   it is not natural to derive symmetric FVM,
   cf. \cite{Hermeline,Li.Y;Shu.S;Xu.Y;Zou.Q2011}
  and references therein.
 %  This fact leads that the common used solver for symmetric linear
%system, such as conjugate gradient method(CG), is no more available
% and new techniques and linear solvers  should be used for the fast
% computation of the FVM solutions \cite{Li.Y;Shu.S;Xu.Y;Zou.Q2011}.
This is important for eigenvalue problems \cite{Dai}, even in dealing
   low order terms.
On the other hand,  symmetric, high-order finite element
   equations are naturally
   defined, by the orthogonal projection in a Hilbert space, \cite{Brenner}.
Such high-order finite element equations can be solved effectively
   by the multigrid method or other fast solvers \cite{Brenner,Scott-Zhang}.
At the end, we correct the finite element solutions by
    bubble functions, to obtain flux-conserving solutions.
Such a correction does not alter the value of the finite element solution on
   the inter-element boundary, i.e.,
    on vertices, edges, and triangles (in 3D).

The rest of manuscript is organized as follows.
In Section 2,  we introduce the finite element method and analyze its
   flux approximation.
In Section 3, we define a post-processing algorithm.
 We will show the flux-conserving property of the processed solution and
   its optimal order of convergence.
In Section 4, we provide numerical examples in 2D and 3D.
Some remarks are made in Section 5.

\vskip .8cm

\section{On flux-conservation property of standard FEM}
Let $\Omega$ be a   bounded,  polygonal domain in ${\mathbb R}^d, d\geq 2$.
We consider  the following second-order elliptic boundary value
    problem
\begin{eqnarray}\lab{eq:2.1}
-\nabla\cdot(\alpha (\b x) \nabla u)
	 =&f  &{\rm in}\ \ \ \Omega,\\ \lab{eq:2.2}
      u=&0 &{\rm on}\ \ \  \Gamma,
\end{eqnarray}
where the boundary $\Gamma=\partial\Omega$ and $\alpha(\b x)$ is a bounded
  and  piecewise  continuous function that is bounded below:
   There exists a constant $\alpha_0>0$ such that
   $\alpha(\b x)\geq \alpha_0$ for almost all $x\in \Omega$.
We choose homogeneous Dirichlet boundary condition here only for
   simplicity of presentation.
The  analysis is similar for other boundary conditions.

Let ${\mathcal T}_h$ be a quasi-uniform, and shape-regular simplicial {\it
   triangulation} on the domain $\Omega$ :
	\a{ \Th = \{ \tau \mid \operatorname{diam}(\tau)\le h <1 \}. }
 With respect to ${\mathcal T}_h$, we define
     an order $k$ finite element subspace
\begin{equation}\lab{trialspace}
   U_h:=\left\{v\in C(\overline{\Omega}): v|_\tau\in {\mathcal P}_k, \ \
    \mbox{for all}\  \tau\in {\mathcal T}_h,\ \
   v|_{\partial\Omega}=0\right\},
\end{equation}
  where ${\mathcal P}_k$ is the
      space of all polynomials of degree equal to or less than $k$.
  Thus, $U_h\subset H_0^1(\Omega)$.
     The finite element solution of \eqref{eq:2.1}
       and \eqref{eq:2.2} is a function $u_h\in U_h$ such that
\begin{equation}\lab{femsolu}
    a(u_h, v_h)=(f,v_h) \quad v_h\in U_h,
\end{equation}
where the bilinear form $a (\cdot,\cdot)$ is defined  by
\[
    a(v,w)=\int_\Omega \alpha(\b x)\nabla v\cdot\nabla w\,  d\b x
	\quad \forall v,w\in H_0^1(\Omega),
\]
and the inner product $(\cdot,\cdot)$
   is defined by
\[
(v,w)=\int_\Omega vw\,  d\b x
	\quad \forall v,w\in  L^2(\Omega).
\]

Unlike the finite volume solutions, the FEM solution $u_h$ in \meqref{femsolu}
    does not satisfy the following conservation law elementwise
\[   \int_\tau f\,d\b x +\int_{\partial\tau}\alpha(\b x)
    \frac{\partial  u_h}{\partial {\bf n}} \, ds \ne 0
   \quad \forall \tau\in \Th.
  \]
 We will define a post-processing method next section so that the
    above flux-error is zero elementwise.
But we will show first that
   the finite element solution does preserve flux locally in
     certain degree.
In this paper, we adopt a notation $\lesssim$, for ``$\le C$",
   where $C$ is a generic, positive constant, independent of the
	grid size $h$, and some other parameters.
  That is,  ``$A\lesssim B$" means that $A$ can be
   bounded by $B$ multiplied by a constant which is independent of
   the parameters which $A$ and $B$ may depend on.

\begin{theorem}\lab{th:flux}
Let $u\in H^{k+1}(\Omega)$ solve \meqref{eq:2.1}--\meqref{eq:2.2} and
  $u_h$ solve \meqref{femsolu}. It holds that
\begin{equation}\lab{flux}
  \sum_{\tau\in\Th} \left|\int_\tau f\,d\b x+
    \int_{\partial\tau}\alpha(\b x)
	\frac{\partial  u_h}{\partial {\bf n}} ds \right|
    \lesssim h^{k-1}|u|_{H^{k+1}},
\end{equation}
where the hidden constant depends only on $\alpha(\b x)$ and
	the shape-regularity of $\Th$.
\end{theorem}
\begin{proof} For all $\tau\in \Th$, by \meqref{eq:2.1} and the divergence
   theorem,
 \[   \int_{\tau}  f \, d\b x =-\int_{\partial\tau} \alpha(\b x)
      \frac{\partial u}{\partial {\bf n}} ds.  \]
By \meqref{femsolu}, the finite element solution has a local flux-error,
\begin{equation}\lab{dflux}
 \int_\tau f\, d\b x + \int_{\partial\tau}\alpha(\b x)
  \frac{\partial  u_h}{\partial {\bf n}} ds
	 =-\int_{\partial\tau} \alpha(\b x)
   \frac{\partial (u-u_h)}{\partial {\bf n}} ds.
\end{equation}
Since $\Th$ is a shape-regular partition,  for all $\tau\in \T$,
  by the trace theorem and the Schwarz  inequality,
\[
  \left\|\frac{\partial (u-u_h)}{\partial {\bf n}}\right\|_{L^2(\partial\tau)}
   \lesssim h_\tau ^{-1/2}\left|u-u_h\right|_{H^1(\tau)}
   +h_\tau ^{1/2} \left|u-u_h\right|_{H^2(\tau)},
\]
where $h_\tau=\operatorname{diam}(\tau) \sim |\tau|^{1/d}$.
The fact that $u\in H^{k+1}(\Omega)$ yields $u-u_h\in H^{k+1}(\tau)$
	for all $ \tau\in \Th$.
By \cite[Theorem 4.14]{Adams}, as $\tau$ is shape-regular, we have
   \a{ \left|u-u_h\right|_{H^2(\tau)} &\lesssim
           h_\tau^{-2+d/2} \left|\hat u-\hat u_h\right|_{H^2(\hat\tau)} \\
	&\lesssim   h_\tau^{-2+d/2}
	\left( \left|\hat u-\hat u_h\right|_{H^1(\hat\tau)}
           +  \left|\hat u-\hat u_h\right|_{H^{k+1}(\hat\tau)} \right)\\
	&\lesssim   h_\tau^{-1}\left|u-u_h\right|_{H^1(\tau)}
	 +  h_\tau^{k-1}\left|u-u_h\right|_{H^{k+1}(\tau)} \\
	&=  h_\tau^{-1}\left|u-u_h\right|_{H^1(\tau)}
	 +  h_\tau^{k-1}\left|u\right|_{H^{k+1}(\tau)}, }
   where $\hat \tau$ is the reference triangle or tetrahedron,
    see \meqref{reference-t}.
 Thus, it follows that
\a{
 \left|u-u_h\right|_{H^2(\tau)}
    &\lesssim   h_\tau ^{-1}\left|u-u_h\right|_{H^1(\tau)}
   +h_\tau ^{k-1}\left|u\right|_{H^{k+1}(\tau)}.
  }
Therefore, by \eqref{dflux}, %after summing over all elements in $\Th$,
\an{ \lab{lf}
\left|\int_\tau f \, d\b x +\int_{\partial\tau}\alpha(\b x)
   \frac{\partial  u_h}{\partial {\bf n}}ds \right|&\le h_\tau ^{1/2}
   \|\alpha(\b x) \|_{L^\infty(\Omega)}
       \left\|\frac{\partial (u-u_h)}
    {\partial {\bf n}}\right\|_{L^2(\partial\tau)}
   \\ \nonumber
   &\lesssim  |u-u_h|_{H^1(\tau)}+ h_\tau^{k}|u|_{H^{k+1}(\tau)}.
    }
Summing up the above inequalities for all $\tau\in\T$ and using the
  optimal order approximation property that
\an{\label{femoptimal}
|u-u_h|_{H^1(\Omega)}\le h^{k}|u|_{H^{k+1}(\Omega)},
}
the inequality \eqref{flux} follows.
\end{proof}

\begin{remark} A similar estimate has been derived in
  \cite{Xu.J;Zou.Q2009} for $P_1$ finite element solutions.
In \cite{Xu.J;Zou.Q2009}, on the dual-cell, one order higher convergence
   is proved for the local flux:
   \a{  \sum_{\b x_i \in \operatorname{vertex}(\Th)}
    \left|\int_{S_{\b x_i}} f\,d\b x+
    \int_{\partial S_{\b x_i}}\alpha(\b x)
	\frac{\partial  u_h}{\partial {\bf n}} ds \right|
    \lesssim h |u|_{H^{2}(\Omega)}, }
   where $S_{\b x_i}=\cup_{\b x_i\in\tau}
	   \{\lambda_i \ge \lambda_{j} \}$,
    see the left figure in Figure \mref{bubble13}.
   Here $\lambda_i$ is the barycentric-centric coordinate on $\tau$ associated
      with vertex $\b x_i$,  cf. \meqref{bubble}.
This is because there is a supercloseness between the $P_1$ finite element
	solution and the finite volume solution, and
   because the flux-error is identically $0$ on each
    vertex-centered control volume $S_{\b x_i}$ for
    the finite volume solution.
Here we give an
  unified estimate for element-wise fluxes of FEM solutions of any order.
\end{remark}

As a direct consequence of \eqref{lf} and \eqref{femoptimal},
  we have the following estimate for the $L^2$ norm of
    numerical fluxes of FEM solutions.

\begin{corollary} Let $u\in H^{k+1}(\Omega)$ solve
  \meqref{eq:2.1}--\meqref{eq:2.2}  and
  $u_h$ solve \meqref{femsolu}. It holds that
\a{
  \left(\sum_{\tau\in\Th} \left|\int_\tau f\,d\b x+
    \int_{\partial\tau}\alpha(\b x)
	\frac{\partial  u_h}{\partial {\bf n}} ds \right|^2\right)^{1/2}
    \lesssim h^{k}|u|_{H^{k+1}(\Omega)}.
    }
\end{corollary}
\begin{proof} By \meqref{lf},
  \a{ \left|\int_\tau f \, d\b x +\int_{\partial\tau}\alpha(\b x)
   \frac{\partial  u_h}{\partial {\bf n}}ds \right|^2
   & \lesssim  |u-u_h|_{H^1(\tau)}^2
	+ h_\tau^{k}|u|_{H^{k+1}(\tau)}^2.
    }
  Combining them on all elements,  we get
   \a{
  \left(\sum_{\tau\in\Th}  \left|\int_\tau f \, d\b x
     +\int_{\partial\tau}\alpha(\b x)
   \frac{\partial  u_h}{\partial {\bf n}}ds \right|^2 \right)^{1/2}
	& \lesssim
    \left( |u-u_h|_{H^1(\Omega)}^2
	+ h ^{k}|u|_{H^{k+1}(\Omega)}^2 \right)^{1/2} \\
	& \lesssim   |u-u_h|_{H^1(\Omega)} + h ^{k}|u|_{H^{k+1}(\Omega)}.
   }
 The corollary is proved by the optimal-order projection property  \eqref{femoptimal} of
    the finite element solution.
\end{proof}

\vskip .8cm

\section{Construction of flux-conserving finite element solutions}

 In this section, we present an algorithm to process
    the FEM solution  so that the flux is elementwise conserving.

Let the $d$-dimensional bubble function be
\begin{equation}\lab{bubble}
   b_\tau=(d+1)^{d+1}\lambda_1\ldots\lambda_{d+1},
\end{equation} on a triangle or tetrahedron
$\tau\in\T$.  Here $\{ \lambda_i, i=1,\ldots, d+1 \}$
    are the barycentric coordinates on $\tau$.
That is,  $\lambda_i(\b x)$ is a linear function defined by
   \a{   \lambda_i(\b x_j) = \delta_{ij} \quad
	\hbox{ at $(d+1)$ vertices $\{\b x_j\}$  of simplex } \ \tau. }
By the divergence-theorem,
  \an{\lab{b-n}  \int_{ \tau} \Delta b_\tau \,d\b x
	= \int_{\partial\tau} \frac{\partial b_\tau}
       {\partial {\bf n}} \, ds  \gtrsim  h^{d-2},  }
  noting that $ {\partial b_\tau}/  {\partial {\bf n}} > 0$ on
	$\partial \tau$ except the $(d+1)$ vertices.
The scaling argument would give
   \an{ \lab{b-s} \|b_\tau\|_{H^1(\tau)}
    \lesssim h^{-1} \|b_\tau\|_{L^2(\tau)} \lesssim h^{d-1},
    \quad \hbox{ in $d$-dimension.}  }

\begin{algo}\lab{a2}
   Given the problem \meqref{eq:2.1}--\meqref{eq:2.2}  and a
	finite element space \meqref{trialspace}.

{\bf Step 1}. Solve the finite element system
\[
    a(u_h,v_h)=(f,v_h),\quad \forall v_h\in U_h
\]
   to obtain the standard
    finite element solution $u_h\in U_h$.

{\bf Step 2}. Correct $u_h$ to obtain a flux-conserving solution:
   \begin{equation}\lab{finsolu}
      \tilde{u}_h=u_h+\sum_{\tau\in\Th}  \gamma_\tau b_\tau,
  \end{equation}
  where $b_\tau$ is defined in \meqref{bubble} and
\an{\lab{gamma}
   \gamma_\tau=\left(\int_{\partial\tau}\alpha (\b x)
	\frac{\partial b_\tau}
       {\partial {\bf n} } ds \right)^{-1}
     \left( -\int_\tau f d\b x -\int_{\partial\tau}
     \alpha(\b x) \frac{\partial u_h}{\partial {\bf n}}ds \right) \quad
	 \tau\in \Th.
   }
\end{algo}

\begin{remark} The post-processed solution $\tilde{u}_h$
 satisfies the local
  flux-conservation property,
\a{  \int_{\partial\tau} \alpha (\b x) \frac{\partial (u-\tilde{u}_h) }
    {\partial {\bf n}}  ds & =
      -  \int_\tau f \, d \b x
     - \int_{\partial\tau} \alpha (\b x) \frac{\partial \tilde{u}_h}
         {\partial {\bf n}}  ds \\
	& = -  \int_\tau f \, d \b x
     - \int_{\partial\tau} \alpha (\b x) \frac{\partial  {u}_h}
         {\partial {\bf n}}  ds - \gamma_\tau \int_{\partial\tau}\alpha (\b x)
	\frac{\partial b_\tau}
       {\partial {\bf n} } ds \\ &=0
 }
for all $\tau\in \T$.
\end{remark}

\begin{remark} We use the lowest-order bubble functions in
	Algorithm \mref{a2}, i.e., degree-3 bubbles in 2D and degree-4
	in 3D.
  We can use nearly any bubble in the flux-correction, even
	non-polynomial bubbles.
  In the analysis below,  we only require a
	bubble to be zero on the boundary
   and to have non-zero total flux on the boundary of an element, cf.
	\meqref{b-n}.
  Although we use the lower order bubble, but it does not deteriorate
   the approximation  of high order finite element.
  That is,  for example,  we can apply a degree-3 bubble correction to
    a 4-th degree finite element solution.
  \end{remark}

We next analyze the convergence property  of the flux-conserving
   solution $\tilde{u}_h$.

\begin{theorem}\lab{th:convergence}
If $u\in H^{k+1}(\Omega)$, then the flux-conserving solution
   $\tilde u_h$ defined in \meqref{finsolu} approximates $u$ in
	the optimal order,
\begin{equation}\lab{h1}
  |u-\tilde{u}_h|_{H^1(\Omega)} \lesssim h^k |u|_{H^{k+1}(\Omega)}
\end{equation}
and
\begin{equation}\lab{l2}
   \|u-\tilde{u}_h\|_{L^2(\Omega)}
       \lesssim h^{k+1} |u|_{H^{k+1}(\Omega)}.
\end{equation}
\end{theorem}
\begin{proof} Since $\alpha (\b x) \ge \alpha_0>0$, for all $\tau\in \T$,
\[
  \int_{\partial\tau} \alpha (\b x)
     \frac{\partial b_\tau}{\partial {\bf n}} ds
      \sim h^{d-2}.
\]
By \meqref{gamma} and estimates in Theorem \ref{th:flux},
\begin{equation}\lab{alpha}
   |\gamma_\tau|\lesssim h ^{d-2}(|u-u_h|_{H^1(\tau)}
      +h ^{k} |u|_{H^{k+1}(\tau)}).
\end{equation}
On the other hand, since $\tau\in \Th$ is shape regular,
\[\|b_\tau\|_{L^2(\tau)}\sim |\tau|^{1/2}= h_\tau ^{\frac{d}{2}}\]
and
\[
|b_\tau|_{H^1(\tau)}\sim h_\tau ^{-1}|\tau|^{1/2}\sim h_\tau^{\frac{d}{2}-1}.
\]
Therefore, the correction part of $\tilde u$ is bounded by
\a{
  \left|\sum_{\tau\in \T}\gamma_\tau b_\tau\right|_{H^1(\Omega)}^2
    &= \sum_{\tau\in \T_h}\gamma_\tau^2\big|b_\tau\big|_{H^1(\tau) }^2\\
   &\lesssim    \sum_{\tau\in \T}\gamma_\tau^2 h_\tau^{d-2}
	 \sim h^{d-2} \sum_{\tau\in \T}\gamma_\tau^2 . }
By \meqref{l2} and \meqref{gamma},
   \a{ \sum_{\tau\in \T}\gamma_\tau^2
   &\lesssim  h^{2k}|u|^2_{H^{k+1}(\Omega)} .
   }
Similarly, by the scaling argument,
\[
\left\|\sum_{\tau\in \T_h}\gamma_\tau b_\tau\right\|_{L^2(\Omega)}^2
    \lesssim h^{2k+2}|u|^2_{H^{k+1}(\Omega)}.
\]
By the triangle inequalities
\[
|u-\tilde{u}_h|_{H^{ 1}(\Omega)}\le |u-u_h|_{H^{ 1}(\Omega)}
     +\left|\sum_{\tau\in \T}\gamma_\tau b_\tau\right|_{H^{ 1}(\Omega)},
\]
\[
\|u-\tilde{u}_h\|_{L^2(\Omega)}\le \|u-u_h\|_{L^2(\Omega)}
    +\left\|\sum_{\tau\in \T}\gamma_\tau b_\tau\right\|_{L^2(\Omega)},
\]
   and the standard finite elements estimates on $u_h$ (cf. \cite{Brenner}),
    the estimates \eqref{h1} and \eqref{l2} follow.
\end{proof}

 The corrected solution $\tilde u_h$,
    defined by \meqref{finsolu}, does not satisfy the original
    finite element equations \meqref{femsolu} any more, though it
    remains as an optimal-order solution.
   We can improve the method by using high-order, orthogonal
     bubble functions.
   Let \an{
   \lab{reference-t}
     \hat \tau=\{ (x,y) \ \mid \ x, y, 1-x-y \ge 0 \} }
   be  the
   reference triangle, i.e., the unit right triangle at the origin.
   Let $\hat b^{(k)} \in {\mathcal P}_k(\hat \tau)$ such that
	\a{ \hat b^{(k)} (\b x) &=0 \quad \hbox{on } \ \partial \hat \tau, \\
	    \int_{\hat \tau} \hat b^{(k)}  \Delta p_{k-2}\, d\b x &=0 \quad
		 \forall p_{k-2}\in {\mathcal P}_{k-2}(\hat \tau), \\
          \int_{\hat \tau} \Delta \hat b^{(k)}  \, d\b x &\ne 0. }
    We note that, by writing $\hat b^{(k)} = b_{\hat \tau} \hat p_{k-3}$,
      there are $\dim {\mathcal P}_{k-3}=(k-1)(k-2)/2$ degrees of freedom
      in constructing  $\hat b^{(k)}$ but only
       $\dim (\Delta {\mathcal P}_{k-2})+1=(k-2)(k-3)/2+1$ constraints
	in the constructions.  For example,  we can let
  \an{ \lab{b4}
	  \hat b^{(k)} (\b x) &= \begin{cases} b_{\hat \tau},
		&k=3, \\
         (3x+3y-2) b_{\hat \tau}, &k=4, \\
        (x^2 -3xy +y^2)  b_{\hat \tau}, &k=5, \\
       (2x^2-6xy+2 y^2 \\ \qquad
      + x^3 -16 x^2y+26 xy^2-6 y^3)b_{\hat \tau}, &k=6, \\
       (x^4-10x^3y +20 x^2y^2 -10 xy^3+y^4 )b_{\hat \tau}, &k=7, \\
       (3x^4 -30 x^3y +60 x^2y^2 -30 xy^3 +3 y^4 \\
        \qquad + 3 x^5-66 x^4y +290 x^3y^2   \\
	\qquad -360 x^2 y^3  + 129 xy^4 -10 y^5)b_{\hat \tau},
           &k=8, \end{cases} }
  where $b_{\hat \tau}$ is defined in \meqref{bubble}.
   By the construction,  the bubble functions are orthogonal to
  the finite element functions in $U_h$, in semi-$H^1$ inner-product:
   \a{ \int_\tau \nabla b_\tau \nabla v_h \, d\b x
	&=  \int_{\hat \tau}  \hat \nabla^T \hat b^{(k+2)}
	(JF)^{-T} (JF)^{-1} \hat \nabla \hat v_h |J| \, d\hat{\b x}
   \\ &= - \int_{\hat \tau}   \hat b^{(k+2)}
	  \hat \nabla \cdot
	( (JF)^{-T} (JF)^{-1} \hat \nabla \hat v_h |J| ) \, d\hat{\b x}
     \\ &= - \int_{\hat \tau}   \hat b^{(k+2)}
              p_{k-2}  \, d\hat{\b x} =0,
    } where $v_h\in U_h$ is a degree-$k$ polynomial on $\tau$
    and $p_{k-2}$ is another degree-$(k-2)$ polynomial on $\hat\tau$.
   Thus, if $b_\tau$ in \meqref{finsolu} is replaced by the
    above orthogonal-bubbles \meqref{b4}, then the processed finite
   element solution still satisfies the original equations \meqref{femsolu}:
   \a{ a(\tilde u_h, v_h) = (\nabla u_h,\nabla v_h)
	+ \sum_{\tau \in \Th} \gamma_\tau (\nabla b^{(k+2)}_{\tau},
	\nabla  v_h) = (\nabla u_h,\nabla v_h)=(f, v_h) }
   for all $v_h\in U_h$.  Here we assumed a constant coefficient
     $\alpha(\b x)\equiv \alpha_0$.
   For variable coefficient $\alpha(\b x)$,  the orthogonal bubbles can
    be constructed element by element.

\begin{figure}[htb]
   \begin{center}\setlength{\unitlength}{1in}
 \begin{picture}(4,1.7)(0,.1)
 \put(0,0.12){\includegraphics[scale=0.23]{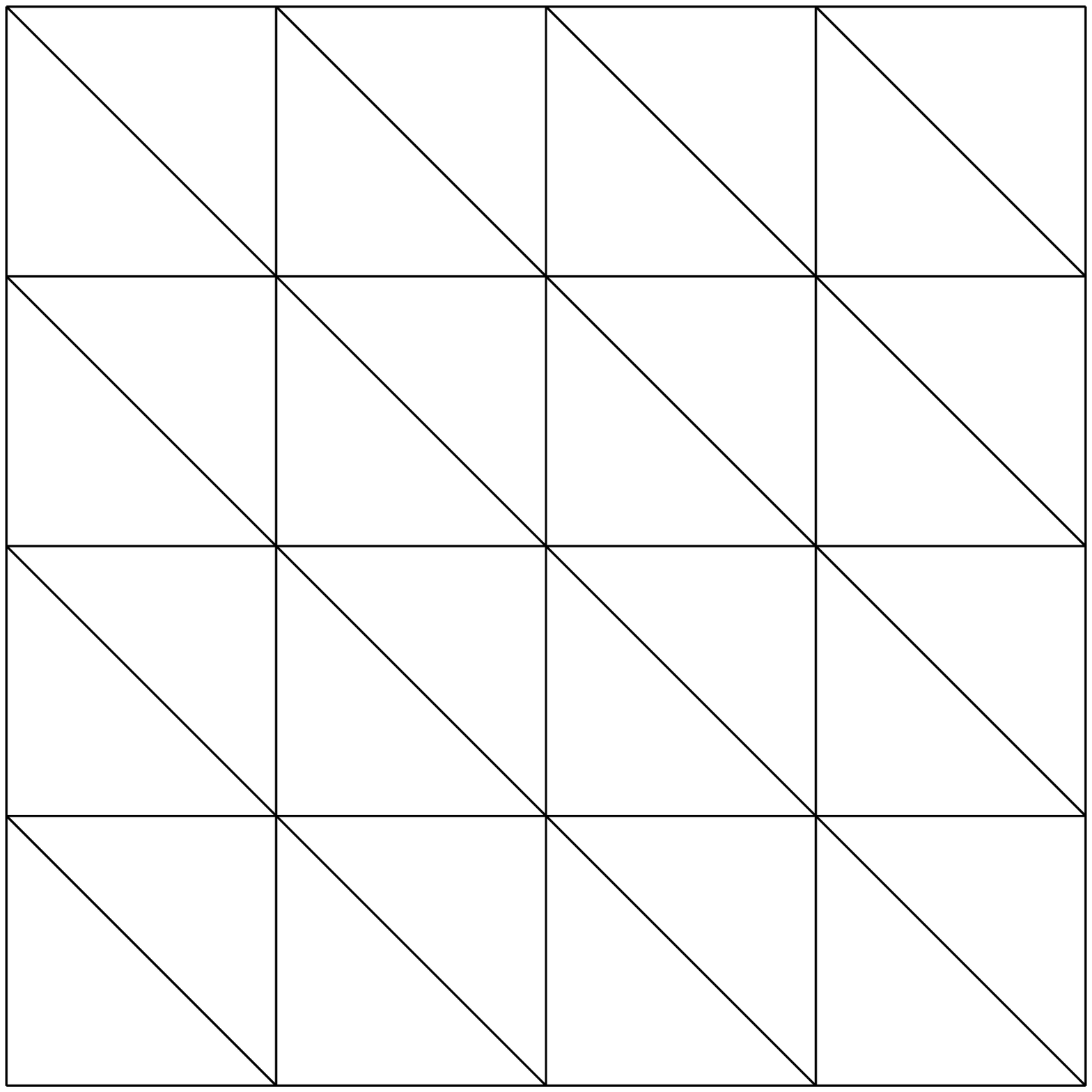}}
 \put(2,0){\includegraphics[scale=0.3]{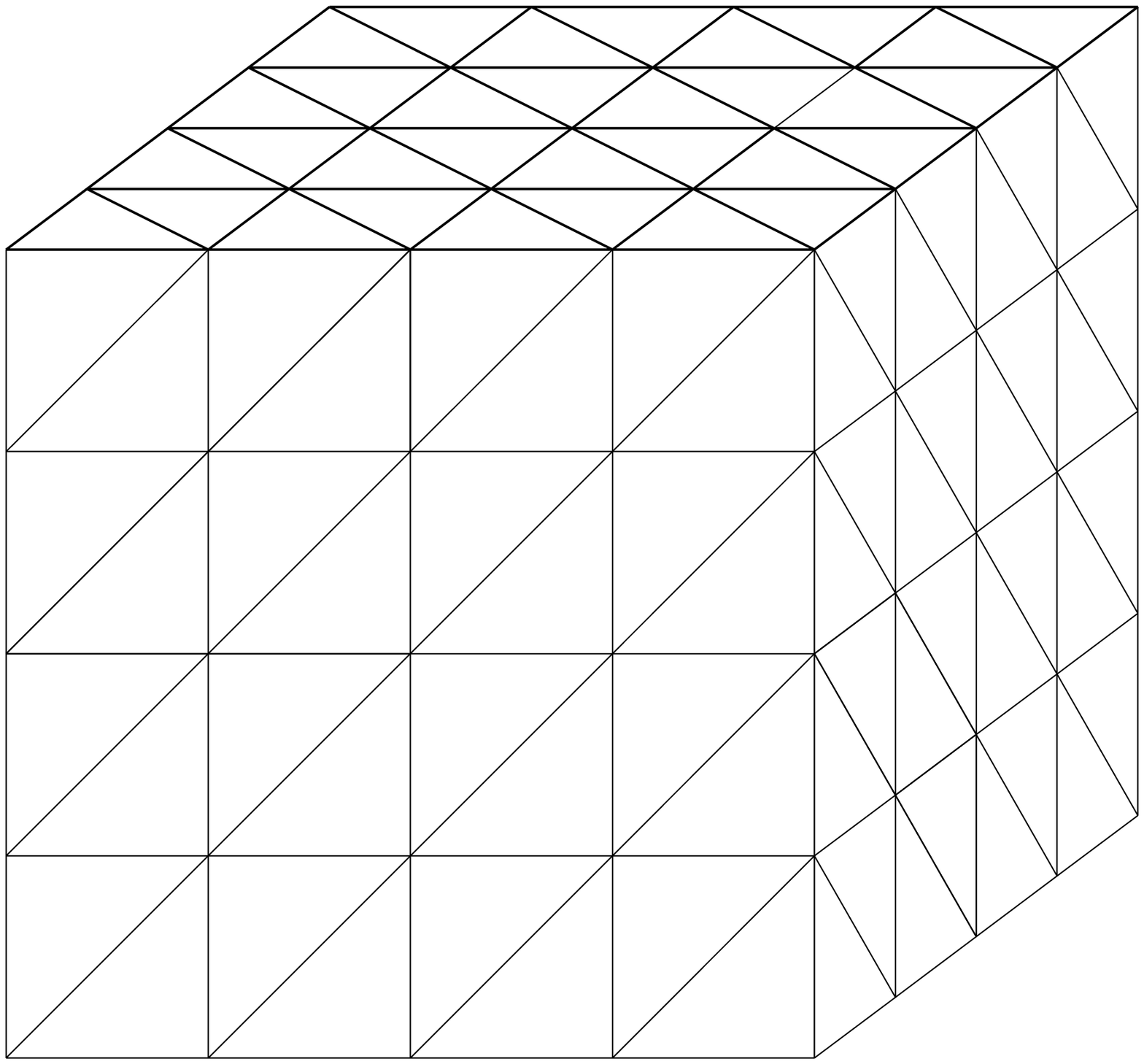}}
 \end{picture}
   \end{center}
\caption{The grid 3 in 2D and 3D $h=1/4$,
     cf. Tables \mref{2p1}--\mref{3p5}.}
    \lab{grids}
\end{figure}

\vskip .8cm

\section{Numerical Experiments}

We provide two numerical tests.  The exact solutions, in 2D and 3D, are
   \an{ \lab{s2} u(x,y) &= 2^8 x^2(1-x)^2 y^2(1-y)^2, \\
	 \lab{s3} u(x,y) &= 2^6 x (1-x)  y (1-y) z(1-z),
   }
 respectively.  We solve the Poisson equation, i.e., \meqref{eq:2.1}
  with $\alpha(\b x)\equiv 1$ on the domain,
	$(0,1)^2$ and $(0,1)^3$.
  The third level grids in 2D and 3D computation are shown in Figure
	\mref{grids}.
  Here we use (multigrid) nested refinement to generate grids,
	cf. \cite{Zhang-tet}.
  We choose a high order polynomial as an exact solution in 2D
	\meqref{s2} only to avoid
    exact finite element solutions when using higher degree elements.

 \begin{table}[htb]
  \caption{ The order of convergence for the $P_1$ element.}
   \lab{2p1}
\begin{center} \vskip -2pt
   \begin{tabular}{c|rr|rr|rr|rr|rr}  %\multispan{3}
 \hline % dim Vh=82   290  1090   4226
   &  $ |e_h |_{H^1}$ &$h^n$&  $ |E_h |_{*}$ &$h^n$
   &  $ |\tilde e_h |_{H^1}$ & $h^n$ &
   $ |  \tilde E_h |_{*}$&$h^n$ & \#cg &dof  \\ \hline
 2& 0.08333 &0.0&  5.000 &0.0& 0.387 &0.7&  0.00000000&0.7 & 2&     9\\
 3& 0.11282 &0.0&  5.102 &0.0& 0.209 &0.9&  0.00000000&0.7& 6&    25\\
 4& 0.03590 &1.7&  5.186 &0.0& 0.107 &1.0& 0.00000000&0.0& 26&    81\\
 5& 0.00953 &1.9&  5.207 &0.0& 0.053 &1.0&  0.00000000&0.4 &55&   289\\
 6& 0.00242 &2.0&  5.211 &0.0& 0.026 &1.0&  0.00000000&0.1& 111&  1089\\
 \hline
 \end{tabular} \end{center}  \end{table}

In the data tables, we use the following notations:
  \a{   e_h & = I_h u - u_h, \\
	E_h & = u - u_h, \\
	\tilde e_h & = I_h u - \tilde u_h, \\
	\tilde E_h & = u - \tilde u_h, }
   where $I_h$ is the nodal value interpolation operator,
     $u_h$ is the finite element solution defined in \eqref{femsolu},
    and $\tilde u_h$ is the post-processed finite element solution
	defined in \eqref{finsolu}.
  Also we use the following norm to measure the error in the total flux:
   \an{\lab{L1} |E_h|_* =\sum_{\tau\in \T_h} \left|
	  \int_{\partial \tau}
	 \frac{\partial ( u-u_h )}{\partial \b n}
	  ds \right|= \sum_{\tau\in \T_h} \left|
	\int_{\tau} f d \b x + \int_{\partial \tau}
	 \frac{\partial  u_h }{\partial \b n}
	  ds \right|. }
   In all our computation,  we use the conjugate gradient method to
   solve the linear systems of finite element equations.
  As the grids are not very fine,  the conjugate gradient method is
   comparable to the optimal order solve, the multigrid method.
   Thus, we list the number of conjugate gradient iterations
      by \#cg in the data tables.

\begin{figure}[htb]
   \begin{center}\setlength{\unitlength}{1in}
 \begin{picture}(4.6,1.3)(0,.1)
 \put(0,0){\includegraphics[scale=0.28]{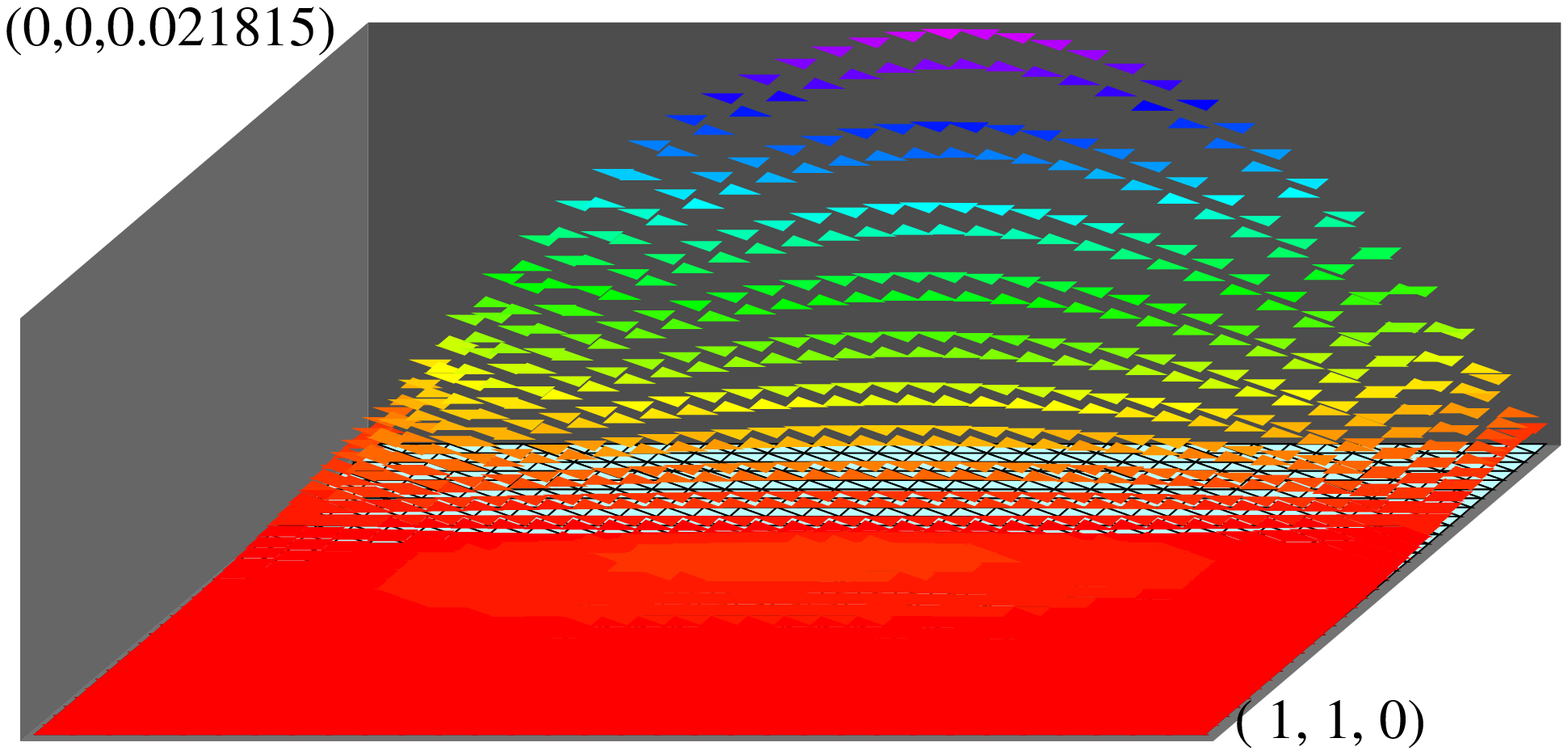}}
 \put(2.4,0){\includegraphics[scale=0.28]{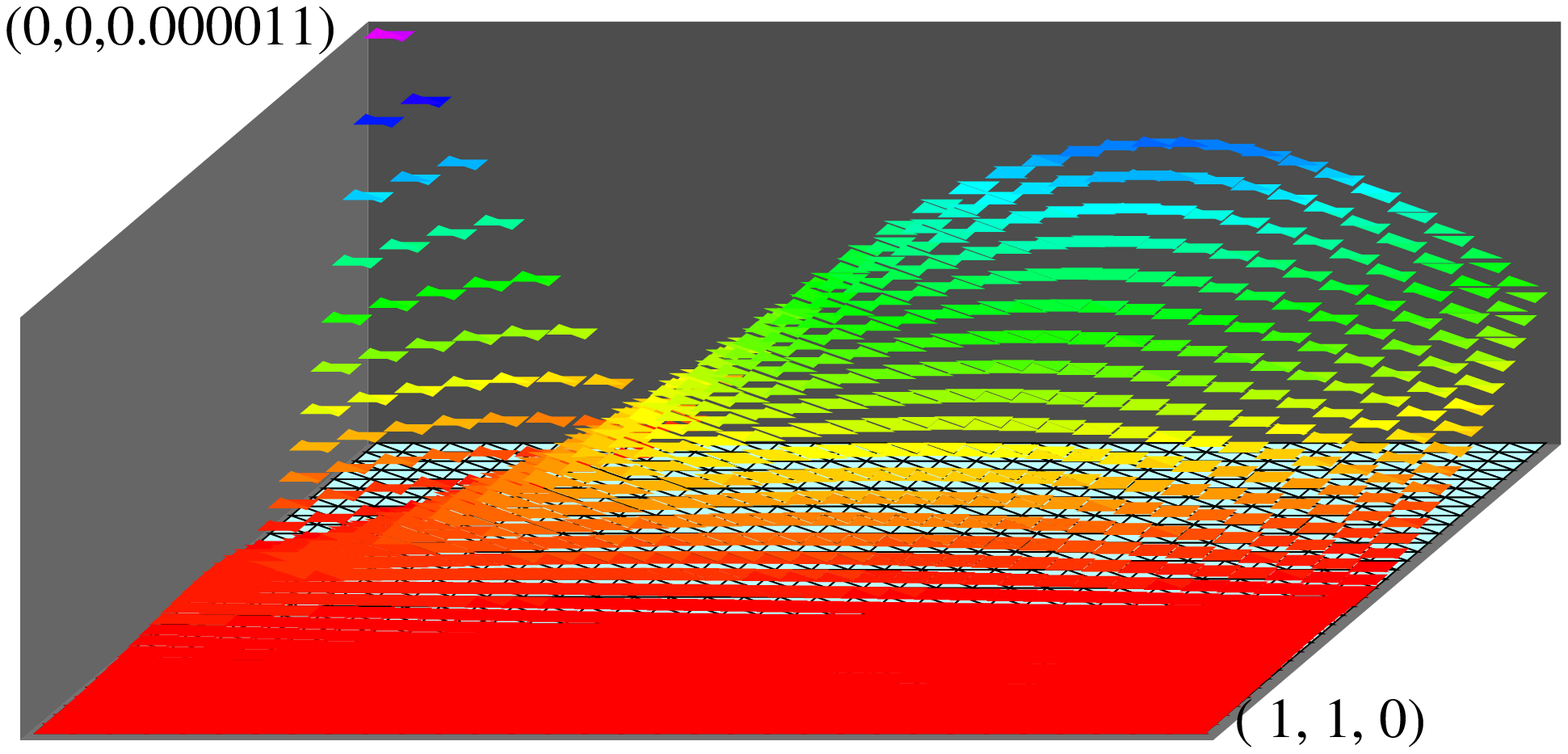}}
 \end{picture}
   \end{center}
\caption{The flux-error \meqref{L1} for $P_1$ and for $P_3$ FEM,
   for \meqref{s2}.}
    \lab{flux-13}
\end{figure}

In Table \mref{2p1},  we can see from the 4th column
      that there is no flux convergence for
   the linear finite element in 2D, confirming Theorem \ref{th:flux}.
However, the total flux error defined in \meqref{L1} is
   completely zero, other than quadrature formula differences,
   by the 8th column data in Table \mref{2p1}.
We note that even though there is no convergence for $P_1$ finite
   element in total flux,  but pointwise flux-error of $P_1$ finite
   element solution does converge to zero when $h$ goes to $0$.
It depends on the way we measure the error.
Since the total inter-element boundary grows at rate $h^{-1}$ in 2D,
    the pointwise flux-error has to converge at a rate higher than $h^1$.
To show this,  we plot the flux-error \meqref{L1} elementwise as a constant
   in Figure \mref{flux-13}.
Where on the left,  we used $P_1$ FEM, and $P_3$ FEM on the right.
The $L^\infty$ flux-error is $O(h)$ for $P_1$ FEM, and
    $O(h^4)$ for $P_3$ FEM.
This is not proved in this paper.

We note that there is a superconvergence in semi-$H^1$ norm for
   $P_1$ finite element, shown in second column  of Table \mref{2p1}.
But after flux-correction,  such a superconvergence no longer holds
	(6-th column).

 \begin{table}[htb]
  \caption{ The errors for the $P_4$ element in 2D.}
   \lab{2p4}
\begin{center} \vskip -2pt
   \begin{tabular}{c|rr|rr||rr|rr }  %\multispan{3}
 \hline % dim Vh=82   290  1090   4226
   &   $ |e_h |_{H^1}$ &$h^n$&
    $ |E_h |_{*}$ &$h^n$
    &  $ | \tilde e_h|_{H^1}$ &$h^n$ &  $  | \tilde e_h^{(6)}|_{H^1}$ &$h^n$      \\ \hline
 1& 0.1291516&0.0&  1.19696&0.0&  0.1837393&0.0&17.35203&0.0  \\
 2& 0.0274134&2.2&  0.30318&2.0& 0.0261514&2.8&  2.26847&2.9 \\
 3& 0.0025223&3.4&  0.04526&2.7& 0.0021648&3.6&   0.17589&3.7 \\
 4& 0.0001742&3.9&  0.00584&3.0&  0.0001448&3.9&  0.01054&4.1 \\
 5& 0.0000112&4.0&  0.00073&3.0& 0.0000092&4.0&  0.00059&4.1  \\
 6& 0.0000007&4.0&  0.00009&3.0&  0.0000005&4.0& 0.00003&4.1 \\
 \hline
  \end{tabular} \end{center}  \end{table}

Next, we use the 2D $P_4$ finite element.
The data are listed in Table \mref{2p4}.
This time, the finite element solution does not have a superconvergence.
So the order of error before and after flux-correction are the
  same, order 4, the optimal order.
However, the finite element does converge in approximating the total
  flux,  but of order 3, shown in the 6th column of Table \mref{2p4},
  conforming Theorem \mref{th:flux} again.
We repeat the computation for $P_4$ FEM but with orthogonal bubbles
  $\hat \tau^{(6)}$ defined in \meqref{b4}.
This time,  $\tilde u_h$ remains as a solution to the original finite
   element equations \meqref{femsolu}.
The convergence is shown in the eighth column of Table \mref{2p4}.
Again, both  the processed finite element solutions have a zero error in
  total flux and both keep the optimal order of FEM approximation.

 \begin{table}[htb]
  \caption{ The order of convergence for the $P_1$ element in 3D.}
   \lab{3p1}
\begin{center} \vskip -2pt
   \begin{tabular}{c|rr|rr|rr|rr|r }  %\multispan{3}
 \hline % dim Vh=82   290  1090   4226
   &  $ \|e_h \|_{L^2}$ &$h^n$&   $ |e_h |_{H^1}$ &$h^n$&
    $ | E_h |_{*}$ &$h^n$
      & \#cg   \\ \hline
 2&     0.11180&0.0&     0.86603&0.0&    10.66667&0.0&    2 \\
 3&     0.06044&0.9&     0.37530&1.2&    10.66667&0.0&    5 \\
 4&     0.01871&1.7&     0.10743&1.8&    10.66667&0.0&   24 \\
 5&     0.00494&1.9&     0.02780&2.0&    10.66667&0.0&   49 \\
 6&     0.00125&2.0&     0.00701&2.0&    10.66667&0.0&   93 \\
 \hline
   &  $ \| \tilde e_h\|_{L^2}$ &$h^n$
   &  $ | \tilde e_h|_{H^1}$ &$h^n$
   &  $ | \tilde E_h|_{*}$ &$h^n$ &dof \\ \hline
 1&     0.05514&0.0&    61.58403&0.0&     0.00000&0.0&     8\\
 2&     0.10502&0.0&    16.41392&1.9&     0.00000&0.0&      27\\
 3&     0.05780&0.9&     4.29661&1.9&     0.00000&0.0&     125\\
 4&     0.01800&1.7&     1.08916&2.0&     0.00000&0.0&    729\\
 5&     0.00477&1.9&     0.27328&2.0&     0.00000&0.0&    4913\\
 6&     0.00121&2.0&     0.06838&2.0&     0.00000&0.0&   35937\\
 \hline
 \end{tabular} \end{center}  \end{table}

Next, we do a 3D computation.
As in 2D, when using linear finite element,  there is no
   convergence in total-flux approximation,  shown in the
	6th column of Table \mref{3p1}.
As shown by Theorem \ref{th:convergence}, the post-processed finite element
  solution converges with the optimal order in $L^2$ and $H^1$
   norm, shown in the second half of Table  \mref{3p1}.
Due to a large flux-error,  the correction bubbles cause a much bigger error
   in energy norm, i.e., semi-$H^1$ norm, see 4th column in
   Tables \mref{3p1}--\mref{3p5}.
  More comments are made below.
However, the correction does not change the FEM solution on all inter-element
   triangles, i.e., $u_h=\tilde u_h$ there.
Note that the error of corrected solution in energy norm is comparable
   to that of the original finite element solution, or even smaller,
   in 2D,  see Tables \mref{2p1}--\mref{2p4}.

 \begin{table}[htb]
  \caption{ The order of convergence for the $P_3$ element in 3D.}
   \lab{3p3}
\begin{center} \vskip -2pt
   \begin{tabular}{c|rr|rr|rr|rr|r }  %\multispan{3}
 \hline % dim Vh=82   290  1090   4226
   &  $ \|e_h \|_{L^2}$ &$h^n$&   $ |e_h |_{H^1}$ &$h^n$&
    $ | E_h |_{*}$ &$h^n$
      & \#cg   \\ \hline
 1&     0.06162&0.0&     0.37386&0.0&     2.50980&0.0&    3 \\
 2&     0.00610&3.3&     0.08622&2.1&     0.94217&1.4&   20 \\
 3&     0.00037&4.1&     0.01297&2.7&     0.21351&2.1&   54  \\
 4&     0.00002&4.1&     0.00175&2.9&     0.05201&2.0&   98 \\
 5&     0.00000&4.0&     0.00023&3.0&     0.01305&2.0&  174 \\
 \hline
   &  $ \| \tilde e_h\|_{L^2}$ &$h^n$
   &  $ | \tilde e_h|_{H^1}$ &$h^n$
   &  $ | \tilde E_h|_{*}$ &$h^n$ &dof \\ \hline
 1&     0.05273&0.0&    14.49518&0.0&     0.00000&0.0&    64\\
 2&     0.00656&3.0&     1.53168&3.2&     0.00000&0.0&  343\\
 3&     0.00040&4.0&     0.09738&4.0&     0.00000&0.0&  2197\\
 4&     0.00002&4.0&     0.00631&3.9&     0.00000&0.0&  15625\\
 5&     0.00000&4.0&     0.00044&3.8&     0.00000&0.0& 117649 \\
 \hline
 \end{tabular} \end{center}  \end{table}

Now, if we use degree 3 finite element in 3D,  we would
   have order 2 convergence for the total flux,  shown
   in Table \mref{3p3}.
We note that the flux-corrected finite element solution also
   converges at order 3 in semi-$H^1$ norm.
Bit it looks like, in Table \mref{3p3}, that
    the order of the flux-corrected finite element solution is
    higher than 3.
It is caused by higher order error in flux.
It seems the bubble term $u_b$ (see \meqref{bubble-E})
    converges to 0 at a higher order
   in this case.
Comparing the errors in column 4 in top and bottom table of Table \mref{3p3},
   we would predict the order of $|\tilde e_h|_{H^1}$ would return to
   $3$ once it catches $|e_h|_{H^1}$.
This can be seen in Table \mref{3p5}.
When using the 5th degree finite element,
   the bubble functions are inside the
  finite element space.
Then it is known from the finite element projection property that
   \an{\lab{bubble-E} |E_h |_{H^1} \le  |\tilde E_h|_{H^1}\le
	 |E_h |_{H^1} +|u_b|_{H^1}, }
  where $u_b=\sum_{\tau\in\T}\gamma_\tau b_\tau$ is the bubble correction term.
Thus, the order of converges of $|\tilde E_h|_{H^1}$ would return
   to that of $|E_h|_{H^1}$
    when the correction $|u_b|_{H^1}$ is no longer dominant,
   see column 4 in the second half of Table \mref{3p5}.

 \begin{table}[htb]
  \caption{ The order of convergence for the $P_5$ element in 3D.}
   \lab{3p5}
\begin{center} \vskip -2pt
   \begin{tabular}{c|rr|rr|rr|rr|r }  %\multispan{3}
 \hline % dim Vh=82   290  1090   4226
   &  $ \|e_h \|_{L^2}$ &$h^n$&   $ |e_h |_{H^1}$ &$h^n$&
    $ | E_h |_{*}$ &$h^n$
      & \#cg   \\ \hline
 1&     0.00927&0.0&     0.11994&0.0&     0.12616&0.0&   12 \\
 2&     0.00014&6.1&     0.00382&5.0&     0.00815&4.0&   77 \\
 3&     0.00000&6.1&     0.00012&5.0&     0.00052&4.0&  159 \\
 4&     0.00000&6.0&     0.00000&5.0&     0.00003&4.0&  290 \\
 5&     0.00000&6.0&     0.00000&5.0&     0.00000&4.0&  505 \\
  \hline
   &  $ \| \tilde e_h\|_{L^2}$ &$h^n$
   &  $ | \tilde e_h|_{H^1}$ &$h^n$
   &  $ | \tilde E_h|_{*}$ &$h^n$ &dof \\ \hline
 1&     0.00927&0.0&     0.73820&0.0&     0.00000&0.0&   216\\
 2&     0.00014&6.1&     0.01242&5.9&     0.00000&0.0&   1331\\
 3&     0.00000&6.1&     0.00022&5.8&     0.00000&0.0&   9261\\
 4&     0.00000&6.0&     0.00000&5.5&     0.00000&0.0&  68921\\
 5&     0.00000&6.0&     0.00000&5.2&     0.00000&0.0&  531441\\
 \hline
 \end{tabular} \end{center}  \end{table}

\vskip .8cm

\section{Concluding Remarks}
 The main motivation in designing FVM is the flux-conservation,
a desirable property in many scientific disciplines such as CFD.
 However, on one side,  the accuracy analysis for high order FVM schemes
  is hard to be established.
On the other hand, the linear algebraic
   system resulting
   from the FVM scheme is often non-symmetric, even if the PDE to be
    solved is self-adjoint.

In this paper, we propose post-processing
   techniques for FEM solutions.
The new solution  preserves both
  local conservation laws and overall accuracy.
Besides,  comparing to the complex construction of FVM schemes
  and high-cost computation of non-symmetric linear systems, any existing finite element code can use our method by adding a simple
subroutine to perform the post-processing.
In this sense, the techniques presented in this paper provide a
   better option, at least for elliptic problems.

Of course  we do not mean that we should abandon the FVM for
  elliptic equations.
In a recent study of high-order FVM schemes
    for 1D elliptic equations \cite{caozhangzou2012},
  it is found that  the derivative of the FVM solutions has higher
   order  superconvergence property  than that of the corresponding
    FEM solutions at some so-called ``optimal stress points".
On the other side, our post processing techniques may even deteriorate the
    superconvergence property of the original FEM solution.
In other words, the FVM schemes may still have their advantages for
   elliptic equations.

\bibliographystyle{plain}

\begin{thebibliography}{10}
% \bibliography{/mypapers/Library/bibliography2nn}

\bibitem{Adams}
   Adams, R. A. Sobolev spaces.
   Pure and Applied Mathematics, Vol. 65.
    Academic Press,
   New York-London, 1975.

\bibitem{BankRRoseD1987}  Bank, R. E.; Rose, D. J.
    Some error estimates for the box method.
   SIAM J. Numer. Anal.  24 (1987),  777--787.

%\bibitem{Barad}
%   Barad,  M.;  Colella, P.
% A fourth-order accurate local refinement method for Poisson's equation.
%   J. Comput. Phys. 209 (2005), 1--18.

\bibitem{Brenner}
  Brenner, S. C.; Scott, L. R.
    The mathematical theory of finite element methods. Third edition.
    Texts in Applied Mathematics, 15. Springer, New York, 2008.

\bibitem{caozhangzou2012}
  Cao, W.;Zhang, Z.;Zou, Q.
  Superconvergence finite volume schemes for 1D elliptic equations. 2011,
  preprint.

\bibitem{Cai1}
   Cai, Z.
   On the finite volume element method.
   Numer. Math.  58  (1991),  no. 7, 713--735.


\bibitem{Cai.Z_Park.M2003}
   Cai, Z.; Douglas, J., Jr.; Park, M.
   Development and analysis of higher order finite volume methods
    over rectangles for elliptic equations.
   Challenges in computational mathematics (Pohang, 2001).
    Adv. Comput. Math.  19  (2003),  no. 1-3, 3--33.

\bibitem{Cai3}
  Cai, Z.; Mandel, J.; McCormick, S.
   The finite volume element method for diffusion equations
    on general triangulations.
  SIAM J. Numer. Anal.  28  (1991),  no. 2, 392--402.

\bibitem{chen}
   Chen, L.
   A new class of high order finite volume methods for
     second order elliptic equations.
   SIAM J. Numer. Anal.  47  (2010),  no. 6, 4021--4043.


\bibitem{ChenWuXu2011} Chen, Z.;  Wu, J.; Xu, Y.
  Higher-order finite volume methods for elliptic boundary value problems.
  Adv. in Comput. Math., in Press.

%\bibitem{Colella}
%   Colella, P.; Dorr, M. R.; Hittinger, J. A. F.; Martin, D. F.
%     High-order, finite-volume methods in mapped coordinates.
%  J. Comput. Phys.  230  (2011),  no. 8, 2952--2976.

 \bibitem{Dai}
    Dai, X.; Gong, X.; Yang, Z.; Zhang, D.; Zhou, A.
    Finite volume discretizations for eigenvalue problems with
      applications to electronic structure calculations.
     Multiscale Model. Simul.  9  (2011),  no. 1, 208--240.


\bibitem{EwingLinLin2002} Ewing, R. E.; Lin, T.;  Lin, Y.
   On the accuracy of finite volume
   element method based on piecewise linear polynomials.
    SIAM J. Numer. Anal. 39 (2002), 1865--1888.

\bibitem{Eymard}
  Eymard, R.; Gallouet, T.; Herbin, R.
  Finite Volume Methods, vol. 7. North Holland, Amsterdam, 2000.

\bibitem{Hackbusch1988}  Hackbusch, W. On first and second
  order box methods.   Computing  41 (1989), 277-296.

\bibitem{Hermeline}
   Hermeline, F. A finite volume method for approximating
      3D diffusion operators on general meshes.
  J. Comput. Phys.  228  (2009),  no. 16, 5763--5786.


\bibitem{LeVeque2002}
  LeVeque, R. J.
   Finite volume methods for hyperbolic problems.
   Cambridge Texts in Applied Mathematics.
   Cambridge University Press, Cambridge, 2002.

\bibitem{LiChenWu2000}   Li,  R.;   Chen, Z.; Wu, W.
  The Generalized Difference Methods for
  Partial Differential Equations  (Numerical Analysis of Finite
  Volume Methods).  Marcel Dikker,  New York, 2000.

\bibitem{Li.Y;Shu.S;Xu.Y;Zou.Q2011}  Li, Y.;  Shu, S.; Xu, Y.; Zou, Q.
   Multilevel preconditioning for the finite volume method.
   Math. Comp., in press.

\bibitem{Liebau1996}
  Liebau, F.
  The finite volume element method with quadratic basis
   functions. Computing 57 (1996), no. 4, 281--299.

\bibitem{Shu}
  Noelle, S.; Xing, Y.; Shu, C.-W.
   High-order well-balanced finite volume
     WENO schemes for shallow water equation with moving water.
    J. Comput. Phys.  226  (2007),  no. 1, 29--58.

%\bibitem{Piller} Piller, M.; Stalio, E.
% Development of a mixed control volume-finite element method
%   for the advection-diffusion equation with spectral convergence.
%  Comput. \& Fluids  40  (2011), 269--279.

\bibitem{PlexousakisZouraris2004}  Plexousakis, M.;  Zouraris, G.
  On the construction and analysis of high order locally conservative
  finite volume type methods for one dimensional elliptic problems,
  SIAM J. Numer. Anal. 42 (2004), 1226--1260.

\bibitem{Scott-Zhang} Scott, L. R.; Zhang, S.
 Higher-dimensional nonnested multigrid methods.
  Math. Comp. 58 (1992), no. 198, 457--466.

 \bibitem{Tian.Chen_1991}
  Tian, M.; Chen, Z.
  Quadratical element generalized differential methods
   for elliptic equations.
     Numer. Math. J. Chinese Univ. 13 (1991), 99--113.

\bibitem{xu} Vogel, A.; Xu, J.; Wittum, G.
    A generalization of the vertex-centered finite volume scheme to
       arbitrary high order.
    Comput. Vis. Sci.  13  (2010),  no. 5, 221--228.

\bibitem{Xu.J;Zou.Q2009}  Xu, J.;  Zou, Q.
  Analysis of linear and quadratic simplicial finite volume
     methods for elliptic
     equations. Numer. Math., 111 (2009),  469--492.


\bibitem{Zhang-tet}  Zhang, S.
 Successive subdivisions of tetrahedra and
    multigrid methods on tetrahedral meshes.
    Houston J. of Math. 21 (1995), 541--556.
\end{thebibliography}

%\end{CJK}
\end{document}